\documentclass[12pt]{article}

\usepackage{amssymb,amsmath,latexsym}
\usepackage{amsthm}
\usepackage[mathscr]{eucal}
\usepackage{setspace}
\usepackage[usenames]{color}
\usepackage{url}

\newtheorem{theorem}{Theorem}[section]
\newtheorem{proposition}[theorem]{Proposition}
\newtheorem{cor}[theorem]{Corollary}
\newtheorem{lemma}[theorem]{Lemma}

\newtheorem{definition}[theorem]{Definition}

\newtheorem{remark}[theorem]{Remark}

\def\theequation{\arabic{section}.\arabic{equation}}

\newcounter{bean}
\newcommand{\benuma}{\setlength{\labelwidth}{.25in}
\begin{list}%
{(\alph{bean})}{\usecounter{bean}}}
\newcommand{\eenuma}{\end{list}}

\newcommand{\la}{\lambda}

\newcommand{\noi}{\noindent}
\newcommand{\bi}{\begin{itemize}}
\newcommand{\ei}{\end{itemize}}
\newcommand{\be}{\begin{enumerate}}
\newcommand{\ee}{\end{enumerate}}
\newcommand{\beqs}{\begin{equation*}}
\newcommand{\eeqs}{\end{equation*}}
\newcommand{\beq}{\begin{equation}}
\newcommand{\eeq}{\end{equation}}
\newcommand{\beqys}{\begin{eqnarray*}}
\newcommand{\eeqys}{\end{eqnarray*}}
\newcommand{\beqy}{\begin{eqnarray}}
\newcommand{\eeqy}{\end{eqnarray}}

\newcommand{\NN}{I\!\!N}
\newcommand{\Z}{\mathbb{Z}}

\newcommand{\R}{I\! \! R}
\newcommand{\ith}{i^{\tiny{\mbox{th}}}}

\newcommand{\ra}{\rightarrow}
\newcommand{\Ra}{\Rightarrow}

\newcommand{\sig}[2]{\sum_{#1}^{#2}}
\newcommand{\inte}[2]{\int_{#1}^{#2}}
\newcommand{\atob}[2]{{#1}^{#2}}
\newcommand{\ab}[3]{{#1}^{#2}_{#3}}
\newcommand{\eps}{\epsilon}

\newcommand{\K}{I\! \! K}

\begin{document}

\title{Non-Markovian state dependent networks in critical loading}

\author{Chihoon Lee\\
Department of Statistics\\
Colorado State University\\
\and Anatolii A. Puhalskii \\
University of Colorado  Denver and \\
Institute for Problems in Information Transmission\\
}

\date{\today}
\maketitle


\begin{abstract}
\noi  We establish heavy traffic limit theorems for
queue-length processes in
 critically loaded single class queueing networks with
state dependent   arrival and service rates.
A distinguishing feature of our model is  non-Markovian state
dependence.   The limit stochastic process
is a continuous-path reflected process on the nonnegative orthant.
We give an application to generalised Jackson networks with
state-dependent rates.
\end{abstract}
\noi {\it Keywords: State dependent networks, non-Markovian networks, diffusion approximation, weak convergence}

\noi {\it   AMS Subject Classifications:} Primary 60F17; secondary 60K25, 60K30, 90B15

\section{Introduction}\label{intro}
Queueing systems with arrival and (or) service rates depending on the system's
state  arise in various application areas which include, among others,
 manufacturing,  storage, service engineering, and communication
and computer networks.
Longer queues may lead to customers being discouraged
 to join the queue, or to faster processing, {e.g.,
   when human servers are involved,}
state-dependent features are  present in congestion
 control protocols in communication networks,  such as TCP
(cf. \cite{ Bek05, ChaSch08, Kushner:2001, MandPat98, Whitt1990} and references
therein for more detail).

In this paper, we
 consider an open network of single server queues
 where the
arrival and service rates depend on the  queue
lengths. More specifically, the network comprises $K$ single-server
stations indexed 1 through $K$\,.    Each station has
an infinite capacity buffer and the jobs are processed according to
the first-in-first-out discipline. The arrivals of jobs at the stations
occur both externally, from the outside and  internally, from the
other stations. Upon service completion at a station, a customer
is either routed to another  station or exits the network.
Any customer entering the
network eventually leaves it.
A distinguishing feature of the model is  non-Markovian state
dependence: the arrival and service processes are assumed to be
time-changed  ``primitive'' processes, where the time change is given by the
integral of the queue-length-dependent rate.

Our goal is to obtain limit theorems in critical loading for the queue
length processes akin to standard
diffusion approximation results available
for generalised Jackson networks, see Reiman \cite{Reiman:1984}.
To this end, we consider a sequence of networks with similar structure
indexed by $n\in\NN$. The limits are taken as $n\to\infty$\,. The
critical loading condition is assumed to hold in the limit.
We show that if  the network primitives satisfy limit theorems with
continuous-path limits,
then the multidimensional
 queue-length processes, when suitably scaled and normalised, converge
 to  a reflected
continuous-path process on  the nonnegative orthant. If the limits of the
primitives are diffusion processes, the limit stochastic
process is a reflected diffusion
with a state dependent  drift and diffusion. We give an application
 to
generalised Jackson networks with state-dependent rates thus providing
an extension of Reiman's \cite{Reiman:1984} results. In addition, we establish
the existence and uniqueness of the solution to the equations
governing the network.

The   results on the  heavy traffic asymptotics in critical loading
for state-dependent rates
 available in the literature
 are mostly confined to the case of diffusion
limits for Markovian models (which assume  state dependent   Poisson arrival
processes and exponential service times), see
 Chapter 8 of Kushner \cite{Kushner:2001},
Mandelbaum and Pats \cite{MandPat98},
  and Yamada
\cite{Yamada95}.   A Markovian closed network with state dependent
rates
 has been considered in
Krichagina \cite{Kri92}.
Some non-Markovian models (with Poissonish
arrivals) have been treated in Section 7 of Yamada \cite{Yamada95}.

A different class of results on diffusion
approximation
 concerns queueing
systems modeled on the many-server queue with a large number of
servers. In such a system the service rate decreases to zero gradually
with the number in the system (whereas in the model considered here it
has a jump at zero, see \eqref{4.3}), so the limit process is an
 unconstrained diffusion, see, Massey,
Mandelbaum, and Reiman \cite{ManMasRei98},
    Pang, Talreja, and Whitt \cite{pang-talreja-whitt}, and references
    therein.
We do not consider those set-ups in this paper.

The exposition is organised as follows.
In the next section, we provide a  precise description of the model
and introduce the  hypotheses required for our main result.
  The statement of the limit theorem for the queue-length process and
  its proof are contained in Section \ref{mainresult}.  In
  Section \ref{example} an application to state-dependent generalised
  Jackson networks is presented. The appendix
 contains a proof of the pathwise queue-length construction underlying
 the definition of the model.

Some notational conventions are in order. We use $\Ra$ to
represent convergence in distribution of random elements with values
in an appropriate metric space,  all vectors are understood as column
vectors, $|x|$ denotes the Euclidean length of
a vector $x$, its components are denoted by $x_i$, unless mentioned
otherwise, superscript $^T$ is used to denote the transpose, $1_A$ stands for the indicator function of an
event $A$\,, $\delta_{ij}$ represents Kronecker's delta, {$\lfloor
a\rfloor$ denotes the integer part of a real number $a$\,.}

\section{The non-Markovian state dependent queueing network}\label{model}
\subsection{The network structure}

   Let $(\Omega,
\mathcal{F}, \mathbf P)$ be a probability space where
  all  random variables  considered in this paper are assumed to be
defined. For the $n$-th network and for
$i\in\K$, where  $\K=\{1,2,\ldots,K\}$, let $A_i^n(t)$ represent  the
cumulative number of customers that arrive at station $i$ from outside
the network during the time interval $[0,t]$, and let $D_i^n(t)$
represent the cumulative number of customers that are served at
station $i$ for the first $t$ units of busy time of that station. Let
$\mathcal{J}\subseteq \K$ represent the set of stations with actual arrivals so
that $A^n_i(t)=0$ if $i\not\in \mathcal{J}$\,.
 We call $A^n=(A_i^n, i\in\K)$ and $D^n=(D^n_i, i\in \K)$, where
 $A^n_i=(A^n_i(t), t\geq0)$ and $D^n_i=(D^n_i(t), t\geq0)$, the
 arrival process and service process for the $n$-th network,
 respectively.  We associate with the stations of the network the
 processes $\Phi_i^n=(\Phi_{ij}^n, j\in \K)$, $i\in\K$, where
 $\Phi^n_{ij}=(\Phi^n_{ij}(m), m=1,2,\ldots)$, and $\Phi^n_{ij}(m)$
 denotes the cumulative number of customers among the first $m$
 customers that depart station $i$ which go directly to station $j$.
 The process $\Phi^n=(\Phi^n_{ij}, i,j\in \K)$ is referred to as the
 routing process.  We consider the processes $A^n_i, D^n_i$ and
 $\Phi^n_i$ as random elements of the respective Skorohod spaces
 $\mathbb{D}([0,\infty),\R), \mathbb{D}([0,\infty),\R)$ and
 $\mathbb{D}([0,\infty),\R^K)$; accordingly, $A^n, D^n$ and $\Phi^n$
 are regarded as random elements of $\mathbb{D}([0,\infty),\R^K),
 \mathbb{D}([0,\infty),\R^K)$ and $\mathbb{D}([0,\infty),\R^{K\times
   K})$, respectively. 
Throughout, $S$ will be used to denote the $K$-dimensional non-negative orthant $\ab{\R}{K}{+}$.

We now introduce the equations which specify  the network.   Let
$\ab{\lambda}{n}{i}$, $\ab{\mu}{n}{i}$, where $i\in\K$,  be Borel
functions mapping  $S$ to $\R_+$,  with $\lambda^n_i(x)=0$ if
$i\not\in\mathcal{J}$, and let $\la^n =
(\la_1^n,\ldots,\la^n_K)$ and
$\mu^n=(\mu^n_1,\ldots,\mu^n_K)$\,. These functions have the meaning
of  state-dependent arrival and service rates.  Let
$\ab{N}{A,n}{i}=(\ab{N}{A,n}{i}(t)\,,t\ge0)$ and
$\ab{N}{D,n}{i}=(\ab{N}{D,n}{i}(t)\,,t\ge0)$
represent nondecreasing $\Z_+$-valued processes with
 trajectories in
$\mathbb{D}([0,\infty),\R)$
 and with $\ab{N}{A,n}{i}(0)=
\ab{N}{D,n}{i}(0)=0$. We define {$
\ab{N}{A,n}{i}(t)=\lfloor t\rfloor$}
if $i\not\in\mathcal{J}$\,. (The latter is but a convenient
convention. Since $\lambda^n_i(x)=0$ if $i\not\in\mathcal{J}$, the
process $N^{A,n}_i$ is immaterial, as the equations below show.)
The state of the network at time $t$ {is represented by
$Q^n(t)=(Q^n_1(t),\ldots, Q^n_K(t))$\,, where $Q^n_i(t)$ represents
the number of customers at station $i$ at time $t$\,. It }is assumed to satisfy a.s. the equations:
\begin{subequations}
  \begin{align}
       Q^n_i(t)=& Q^n_{i}(0)
+ A^n_i(t)+B^n_i(t)-D^n_i(t), \label{4}\\  A_i^n(t)=& \ab{N}{A,n}{i}\left(\inte{0}{t}\ab{\lambda}{n}{i}(\ab{Q}{n}{}(s))ds\right), \label{4.1}\\
B_i^n(t)=& \sig{j=1}{K}\Phi^n_{ji}\bigl(D^n_j(t)\bigr), \label{4.2}
\\ D^n_i(t)=& \ab{N}{D,n}{i}\left(\inte{0}{t}\ab{\mu}{n}{i}(\ab{Q}{n}{}(s))1_{\{\ab{Q}{n}{i}(s)>0\}}ds
\right),\label{4.3}
  \end{align}
\end{subequations}
 where $t\geq0$ and $i\in\K$. The
 random quantities in \eqref{4}--\eqref{4.3} have the following
 interpretation: $Q^n_i(0)\in {\Z_+^K}$ is the initial queue length at station
 $i$; $A^n_i(t), B^n_i(t), D^n_i(t)$ represent the cumulative number of exogenous arrivals at station $i$  during the time interval $[0,t]$, the cumulative number of endogenous arrivals at station $i$  during the time interval $[0,t]$, and the cumulative number of departures from station $i$  during the time interval $[0,t]$, respectively.

\subsection{Assumptions on the network primitives}\label{assumption}

Let $P=(p_{ij},\,i,j\in\K)$ be a substochastic matrix, $R=I-P^T$,
and $p_i=(p_{ij},\,j\in\K)$\,.
We denote
\begin{align*}
\overline{Q}^n(0)=&\frac{Q^n(0)}{\sqrt n},&
 \overline{N}^{A,n}_i(t)=&
\frac{N^{A,n}_i(nt)-nt}{\sqrt n}, \\
\overline{N}^{D,n}_i(t)=&\frac{N^{D,n}_i(nt)-nt}{\sqrt n},&
\overline{\Phi}^n_i(t)=&
\frac{\Phi^n_i({\lfloor nt\rfloor})-p_int}{\sqrt n}\,,\\
 \overline{N}^{A,n}_i=& (\overline{N}^{A,n}_i(t),\,t\ge0),&
\overline{N}^{A,n}=&(\overline{N}^{A,n}_i,\,i\in\K),\\
\overline{N}^{D,n}_i=& (\overline{N}^{D,n}_i(t),\,t\ge0),&
\overline{N}^{D,n}=&(\overline{N}^{D,n}_i,\,i\in\K),\\
\overline{\Phi}^{n}_i=& (\overline{\Phi}^{n}_i(t),\,t\ge0),&
\overline{\Phi}^{n}=&(\overline{\Phi}^{n}_i,\,i\in\K)\,.
\end{align*}
We will need the following conditions.
\begin{itemize}
\item[(A0)] For each $n\in\NN$ and each  $i\in\mathcal{J}$,
 $\limsup_{t\to\infty}N^{A,n}_i(t)/t<\infty$\,a.s.
\item[(A1)] The spectral radius of matrix $P$ is strictly less than 1.

\item[(A2)] For each  $i\in\K$,
$$ \sup_{n\in\NN}\sup_{x\in S}\left(\frac{\lambda^n_i(nx)}{n(1+|x|)} + \frac{\mu_i^n(nx)}{n(1+|x|)} \right)<\infty.$$
\item[(A3)] There exist continuous
functions $\lambda_i(x)$ and $ \mu_i(x)$
such that $$\frac{{\lambda}^n_i(nx)}{{n}}\ra \lambda_i(x),\quad
\frac{{\mu}^n_i(nx)}{{n}}\ra \mu_i(x)$$ uniformly on compact
subsets of $S$, as $n\ra\infty$. Furthermore, for $x \in S$,
$$\lambda(x)-R\mu(x)=0.$$
\item[(A4)]
There exists a {Lipschitz}-continuous function $a(x)$ such that
\begin{equation*}
  \frac{1}{\sqrt{n}}\,(\lambda^n(\sqrt{n}x)-R\mu^n(\sqrt{n}x))\to a(x)
\end{equation*}
as $n\to\infty$ uniformly on compact subsets of $S$.
\item[(A5)] As $n\to\infty$,
\begin{align*}
  (\overline{Q}^n(0),
 \overline{N}^{A,n},
\overline{N}^{D,n},
\overline{\Phi}^n) \Ra(X_0, W^A,W^D, W^\Phi)
\end{align*} where $X_0$ is a  random $K$-vector,
$W^A$,  $W^D$, and $W^\Phi$ are continuous-path stochastic processes
with trajectories in  respective spaces
$\mathbb{D}([0,\infty),\R^K)$, $\mathbb{D}([0,\infty),\R^K)$, and
$\mathbb{D}([0,\infty),\R^{K\times K})$\,.
\end{itemize}
Condition (A0) is  needed to ensure the existence of a unique strong
solution to the system of equations \eqref{4}--\eqref{4.3}, see Lemma
\ref{exist}. It is almost a consequence of condition (A5) in that the latter
implies that $\lim_{n\to\infty}N^{A,n}_i(nt)/(nt)=1$ in probability.
Part (A1) is essentially an assumption  that the network is open.
It  implies  the existence of a  regular Skorohod map associated with
the network data (see Proposition \ref{cond2.1}) which is a key
element of the proof of the main result.   The requirement
$\lambda(x)=R\mu(x)$ in {(A3)} together with condition (A4) defines a
\emph{critically loaded heavy traffic}  regime.  Condition (A5) is the
assumption on the primitives. The components of $W^A$ corresponding to
$i\not\in \mathcal{J}$ vanish.
Conditions (A2)--(A4) are fulfilled if  the following expansions  hold:
$\la^n(x) = n\la_1(x/{n})+\sqrt{n}\la_2(x/{\sqrt n})$ and $\mu^n(x) =
n\mu_1(x/{n})+\sqrt{n}\mu_2(x/{\sqrt n})$, where $\la_1$, $\lambda_2$,
$\mu_1,$ and $\mu_2$ are nonnegative bounded continuous functions such that
$\la_1(x) = R\mu_1(x)$\,. If the above functions are constant, then
one obtains the standard critical loading condition that
$(\lambda^n-R\mu^n)/\sqrt{n}\to\lambda_2-\mu_2$ as $n\to\infty$, cf. Reiman \cite{Reiman:1984}.

Most of the results on diffusion approximation in critical loading
 (see, e.g., Harrison and Reiman
\cite{Harrison:Rieman:1981}, Kushner
\cite{Kushner:2001}) formulate the  heavy traffic condition
in terms of rates that are $\mathcal{O}(1)$ and then consider scaled
processes where time is scaled up by a factor of $n$ while space is
scaled down by a factor of $\sqrt{n}$.  In the scaling considered here
(cf.  Mandelbaum and Pats \cite{MandPat98}, Yamada \cite{Yamada95})
 the time parameter is left
unchanged and the factor of $n$ is absorbed in the arrival and service
rates. This is more convenient notationally, however, in the
application to generalised Jackson networks in Section \ref{example} we
{work with} the conventional scaling.

\begin{lemma}\label{exist}
Let conditions (A0) and (A2) hold. Then equations
\eqref{4}--\eqref{4.3} admit
 a unique strong solution $Q^n$,
which is a $\mathbb Z^K_+$-valued stochastic process.
\end{lemma}  The proof is provided in the appendix.  Intuitively, the
linear growth condition (A2) on the total arrival rate, together with
the asymptotic bounds (A0), {implies} that the
sample paths of the $Q^n_i$ are nonexplosive.
\section{Main results}\label{mainresult}
We assume conditions (A0)--(A5) throughout this {section}.
We introduce the ``centered'' processes as follows: For $i \in \K$ and
$t\ge0$,
\begin{subequations}
  \begin{align}
    \label{eq:1}
    \ab{M}{n}{i}(t)=& M^{A, n}_i(t)+M^{B, n}_i(t)-M^{D, n}_i(t),
\intertext{ where }
    \label{eq:1a}  M^{A, n}_i(t)=& \ab{N}{A,n}{i}\left(\inte{0}{t}\ab{\lambda}{n}{i}(\ab{Q}{n}{}(s))ds\right)-\inte{0}{t}\ab{\lambda}{n}{i}(\ab{Q}{n}{}(s))ds, \\    \label{eq:1b}
M^{B, n}_i(t)=&
\sig{j=1}{K}\left(\Phi^n_{ji}\bigl(D^n_j(t)\bigr)-p_{ji}D^n_j(t)\right), \intertext{and}
    \label{eq:1c} M^{D, n}_i(t)=& \ab{N}{D,n}{i}\left(\inte{0}{t}\ab{\mu}{n}{i}(\ab{Q}{n}{}(s))1_{\{\ab{Q}{n}{i}(s)>0\}}ds
\right) -
\inte{0}{t}\ab{\mu}{n}{i}(\ab{Q}{n}{}(s))1_{\{\ab{Q}{n}{i}(s)>0\}}ds\notag\\
+&\sum_{j=1}^Kp_{ji}\Bigl(\ab{N}{D,n}{j}\left(\inte{0}{t}\ab{\mu}{n}{j}(\ab{Q}{n}{}(s))1_{\{\ab{Q}{n}{j}(s)>0\}}ds
\right) -
\inte{0}{t}\ab{\mu}{n}{j}(\ab{Q}{n}{}(s))1_{\{\ab{Q}{n}{j}(s)>0\}}ds\Bigr)\,.
  \end{align}
\end{subequations}

We can rewrite the evolution \eqref{4} as \beqs\label{121} \ab{Q}{n}{i}(t)=
\ab{Q}{n}{i}(0) + \inte{0}{t}
\left[\ab{\lambda}{n}{i}(\ab{Q}{n}{}(s))+ \sig{j=1}{K} p_{ji}
\ab{\mu}{n}{j}(\ab{Q}{n}{}(s)) - \ab{\mu}{n}{i}(\ab{Q}{n}{}(s))
\right]ds +  \ab{M}{n}{i}(t) + [R \ab{Y}{n}{}(t)]_i,
\eeqs where  $Y^n(t)=(Y^n_i(t)\,,i\in\K)$
and \begin{equation}
  \label{eq:11}
\ab{Y}{n}{i}(t)= \inte{0}{t}1_{\{\ab{Q}{n}{i}(s)=0\}}
\ab{\mu}{n}{i}(\ab{Q}{n}{}(s))\,ds, \quad i\in\K,
\end{equation}
Note that $(\ab{Y}{n}{i}(t), t\geq0)$ is a continuous-path
non-decreasing process with $Y^n_i(0)=0$, which  increases only when $\ab{Q}{n}{i}(t)=0$, i.e., $\inte{0}{\infty}1_{\{\ab{Q}{n}{i}(t)\neq0\}}d\ab{Y}{n}{i}(t)=0$ a.s.   Set
\begin{equation}
  \label{eq:2}
a^n(x)= \ab{\lambda}{n}{}(x)-R\ab{\mu}{n}{}(x).
\end{equation}
Then the state evolution can be expressed succinctly by the following vector equation: \beq\label{vec11} \ab{Q}{n}{}(t) = \ab{Q}{n}{}(0) + \inte{0}{t}\ab{a}{n}{}(\ab{Q}{n}{}(s))ds + \ab{M}{n}{}(t) + R\ab{Y}{n}{}(t), \quad t\geq0. \eeq The latter dynamic can equivalently be described in terms of a Skorohod map as described below.

\begin{definition}\label{SP} Let $\psi\in \mathbb{D}([0,\infty),\R^K)$
  be given with $\psi(0)\in S$. Then the pair
$(\phi,\eta)\in \mathbb{D}([0,\infty), \R^K)\times \mathbb{D}([0,\infty),\R^K)$ solves the Skorohod problem for $\psi$ with respect to $S$ and $R$ if  the following hold:
\begin{itemize}
\item [\rm{(i)}] $\phi(t)=\psi(t)+R\eta(t)\in S$, for all $t\geq0;$
\item[\rm{(ii)}]  for $i\in\K,$
\rm{(a)} $\eta_i(0)=0,$ \rm{(b)} $\eta_i$ is non-decreasing, and
\rm{(c)} $\eta_i$ can increase only when $\phi$ is on the $\ith$
face of $S$, that is,
$\int_0^{\infty}1_{\{\phi_i(s)\neq0\}}d\eta_i(s)=0.$
\end{itemize}
\end{definition}
\noi 
Let $\mathbb{D}_S([0,\infty),\R^K)=\{\psi\in 
\mathbb{D}([0,\infty),\R^K):\psi(0)\in S\}$.
{If the Skorohod problem has a unique solution on a
  domain $D\subset
\mathbb{D}_S([0,\infty),\R^K)$,
we define the Skorohod map $\Gamma$ on $D$} by $${\rm{\Gamma}}(\psi)= \phi\,.$$
The following result (see Dupuis and Ishii \cite{Dupuis:Ishii:1991},
Harrison and
Reiman \cite{Harrison:Rieman:1981})
yields the regularity
of the Skorohod map and is a consequence of Assumption (A1).
\begin{proposition}\label{cond2.1}
{The Skorohod map $\Gamma$ is well defined on 
$\mathbb{D}_S([0,\infty),\R^K)$ and} is Lipschitz continuous in the following sense:
There exists a constant $L>0$ such that for all $T>0$ and
 $\psi_1,\psi_2 \in \mathbb{D}_S([0,\infty),\R^K)$,
$$\label{sm}\sup_{t\in[0,T]}|{\rm{\Gamma}}(\psi_1)(t)-{\rm{\Gamma}}(\psi_2)(t)|
\le L\sup_{t\in[0,T]}|\psi_1(t)-\psi_2(t)|.$$
\end{proposition}
{As a consequence, both $\phi$ and $\eta$ are
  continuous functions of $\psi$. (Note that matrix $R$ is invertible
  under the hypotheses.)}

\noi The dynamic in \eqref{vec11} can now be equivalently described
in terms of the Skorohod map as follows:
\beq\label{2.7st}
Q^n(t)={\rm{\Gamma}}\left(Q^n(0)+\inte{0}{\cdot}\ab{a}{n}{}(\ab{Q}{n}{}(s))ds + \ab{M}{n}{}(\cdot)\right)(t), \quad \mbox{for } t\geq0.\eeq

The Lipschitz continuity of the Skorohod map
and of the function $a(x)$ imply that the equation
\beq\label{sde}X(t)=\Gamma\bigl( X_0 +   \inte{0}{\cdot}\ab{a}{}{}(\ab{X}{}{}(s))ds +
 M(\cdot)\bigr)(t), \eeq
where
\begin{equation}
  \label{eq:3}
M_i(t)=W^A_i(\lambda_i(0)t)+\sum_{j=1}^KW^\Phi_{ji}(\mu_j(0)t)
-\sum_{j=1}^K(\delta_{ij}-p_{ji})W^D_j(\mu_j(0)t)\,,
\end{equation}
has a unique strong {solution}.

We now state the main result of this paper.
For $t\geq0$ and $i\in \K$, let  $\atob{X}{n}_i(t)=Q^n_i(t)/\sqrt n$.
We also define $X=(X(t),\,t\ge0)$,  and
$X^n=((X^n_i(t),\,i=1,2,\ldots,{K}),\,t\ge0)$\,.
\begin{theorem}\label{main}
Let conditions (A0)--(A5) hold. Then
 $\atob{X}{n}\Ra X$, as $n\to\infty$\,.
\end{theorem}

To prove this theorem, we first establish certain tightness
results.  Recall that a
sequence $V^n$ of stochastic processes with trajectories in a Skorohod
space is said to be  $\mathbb C$-tight if the sequence of the laws of
the $V^n$ is tight,
and if all limit points of the sequence of the laws of the $V^n$
 are laws of continuous-path processes
 (see, e.g., Definition 3.25 and Proposition 3.26 in Chapter VI of
Jacod and Shiryaev {\cite{MR2003j:60001}}).

\begin{lemma}\label{lem1}
The sequence of processes $(M^{n}(t)/\sqrt{n},\,t\ge0)$ is $\mathbb{C}$-tight.
\end{lemma}
\begin{proof}
By \eqref{4} -- \eqref{4.3},
\begin{multline*}
\sum_{i=1}^K Q^n_i(t)\le\sum_{i=1}^K Q^n_i(0)+\sum_{i=1}^K A^n_i(t)=
\sum_{i=1}^K Q^n_i(0)+\sum_{i=1}^K \ab{N}{A,n}{i}\left(\inte{0}{t}\ab{\lambda}{n}{i}(\ab{Q}{n}{}(s))ds\right)\,.
\end{multline*}
Therefore, for suitable $H>0$, on recalling (A2) and denoting
$Z^n_i(t)=Q^n_i(t)/n$,
\begin{multline*}
  \sum_{i=1}^K Z^n_i(t)\le
\sum_{i=1}^K Z^n_i(0)+
\sum_{i=1}^K
\sup_{y\ge n}
\frac{1}{y}\,\ab{N}{A,n}{i}(y)
\left(1+\inte{0}{t
}\frac{1}{n}\ab{\lambda}{n}{i}(n\ab{Z}{n}{}(s))ds\right)\\\le
\sum_{i=1}^K Z^n_i(0)+
\sum_{i=1}^K
\sup_{y\ge n}
\frac{1}{y}\,\ab{N}{A,n}{i}(y)
\left(1+H\inte{0}{t}(1+\sum_{i=1}^K\ab{Z}{n}{i}(s))\,ds\right){.}
\end{multline*}
By  Gronwall's inequality (cf. p.498 in
Ethier and Kurtz \cite{EthKur86}),
\begin{equation*}
  \sum_{i=1}^K Z^n_i(t)\le
\bigl(\sum_{i=1}^K Z^n_i(0)+
\sum_{i=1}^K \sup_{y\ge n}\frac{1}{y}\,\ab{N}{A,n}{i}(y)
(1+Ht)\bigr)\exp\bigl(H
\sum_{i=1}^K
\sup_{y\ge n}
\frac{1}{y}\,\ab{N}{A,n}{i}(y)t
\bigr){.}
\end{equation*}
By (A5), $\ab{N}{A,n}{i}(y)/y\to 1$ in probability as $y\to\infty$ and
$n\to\infty$ and $\sum_{i=1}^K Z^n_i(0)\to0$ in probability as $n\to\infty$\,. Therefore,
\begin{equation}
  \label{eq:10}
      \lim_{r\to\infty}\limsup_{n\to\infty}\mathbf{P}(\sup_{s\le t}\sum_{i=1}^K Z^n_i(s)>r)=0\,.
\end{equation}
It follows by (A2) that
\begin{equation}
  \label{eq:8}
    \lim_{r\to\infty}\limsup_{n\to\infty}\mathbf{P}(\int_0^t
\bigl(\frac{1}{n}\,\lambda^n_i(Q^n_i(s))
+\frac{1}{n}\,\mu^n_i(Q^n_i(s))\bigr)\,ds>r)=0
\end{equation}
and that, for $\delta>0$, $\epsilon>0$, $T>0$,
\begin{equation}
  \label{eq:9}
  \lim_{\delta\to0}\limsup_{n\to\infty}
\mathbf{P}\bigl(\sup_{t\in[0,T]}\int_{t}^{t+\delta}
\bigl(\frac{1}{n}\,\lambda^n_i(Q^n_i(s))
+\frac{1}{n}\,\mu^n_i(Q^n_i(s))\bigr)\,ds>\epsilon\bigr)=0\,.
\end{equation}

We have that, for $\gamma>0$, $\delta>0$, $\epsilon>0$, $T>0$, and $r>0$,
\begin{multline*}
  \mathbf{P}(\sup_{\substack{s,t\in[0,T]:\\
|s-t|\le \delta}}|\frac{1}{\sqrt{n}}M^{A,n}_i(t)-
\frac{1}{\sqrt{n}}M^{A,n}_i(s)|>\gamma)\le
\mathbf{P}\bigl(\int_0^T\frac{1}{n}\,\lambda^n_i(Q^n_i(s))
\,ds>r\bigr)\\+
\mathbf{P}\bigl(\sup_{t\in[0,T]}\int_{t}^{t+\delta}
\frac{1}{n}\,\lambda^n_i(Q^n_i(s))
\,ds>\epsilon\bigr)+
  \mathbf{P}(\sup_{\substack{s,t\in[0,r]:\\
|s-t|\le \epsilon}}|\overline{N}^{A,n}_i(t)-
\overline{N}^{A,n}_i(s)|>\gamma)\,.
\end{multline*}
By (A5), \eqref{eq:8}, and \eqref{eq:9},
\begin{equation*}
\lim_{\delta\to0}\limsup_{n\to\infty}  \mathbf{P}(\sup_{\substack{s,t\in[0,T]:\\
|s-t|\le \delta}}|\frac{1}{\sqrt{n}}M^{A,n}_i(t)-
\frac{1}{\sqrt{n}}M^{A,n}_i(s)|>\gamma)=0\,.
\end{equation*}
Hence, the sequences of processes $(M^{A,n}_i(t)/\sqrt{n},\,t\ge0)$ are $\mathbb{C}$-tight. A similar argument shows that the sequences of processes $(M^{D,n}_i(t)/\sqrt{n},\,t\ge0)$ and $(M^{\Phi,n}_i(t)/\sqrt{n},\,t\ge0)$ are $\mathbb{C}$-tight, so the sequence of processes $(M^{n}(t)/\sqrt{n},\,t\ge0)$ is $\mathbb{C}$-tight.
\end{proof}

\noindent Next, we identify the limit points of
$\overline{M}^{n}=(M^n(t)/\sqrt{n},\,t\ge0)$.
\begin{lemma}\label{lem2}
The sequence of processes $\overline{M}^{n}$ converges in
distribution, as $n\to\infty$, to $M$\,.
\end{lemma}
\begin{proof}
From Lemma \ref{lem1}, $M^n(t)/n\to 0$ in probability uniformly over bounded intervals.
By (A2) and \eqref{eq:2}, for some $H'$, for all $n$ and $x$,
$  |a^n(nx)|\le H'n(1+|x|)$\,.
By \eqref{eq:10}, the sequence of processes
$(\inte{0}{t}(1/n)\ab{a}{n}{}(\ab{Q}{n}{}(s))ds,\,t\ge0)\,\,\,
\mbox{is }\,\, \mathbb{C}\mbox{-tight}.$ By \eqref{vec11}, the fact
that $M^n(t)/n\to0$ in probability {uniformly} on bounded intervals,
Prohorov's theorem, and the
continuity of the Skorohod map, the sequence of processes
$(Q^n(t)/n,\,t\ge0)$ is $\mathbb{C}$-tight and every its limit in
distribution $(q(t),\,t\ge0)$ satisfies the equation
\begin{equation*}
  q(t) ={\rm{\Gamma}}\left(\inte{0}{\cdot}(\lambda(q(s))-
R\mu(q(s))) ds\right)(t)\,.
\end{equation*}
Since by (A3), $\lambda(x)-R\mu(x)=0$, we must have that $q(t)=0$,
which implies that the {sequence $Q^n_i(t)/n$\,
tends} to zero as $n\ra\infty$ in probability uniformly on bounded
intervals.
Since $Y^n$ is expressed as a continuous function of $(Q^n,M^n)$, we have that
 $Y^n(t)/n\to0$ in probability uniformly
over bounded intervals, so by \eqref{eq:11}, for $i\in\K$,

\begin{equation}
  \label{eq:7}
          \frac{1}{n}
\inte{0}{t}\mu^n_i(Q^n(s))1_{\{Q^n_i(s)=0\}}ds \ra 0
 \,\mbox{ in probability as }\, n\ra\infty.
\end{equation}
We also have by (A3) that
\begin{subequations}
  \begin{align}
\label{eq:6}
      \frac{1}{n}
\inte{0}{t}\lambda^n_i(Q^n(s))\,ds \ra \lambda_i(0)t
 \,\mbox{ in probability as }\, n\ra\infty\\
\intertext{and}
    \label{eq:6a}
      \frac{1}{n}
\inte{0}{t}\mu^n_i(Q^n(s))\,ds \ra \mu_i(0)t
 \,\mbox{ in probability as }\, n\ra\infty\,.
  \end{align}
\end{subequations}

Since by (A5), $N^{D,n}_i(nt)/n\to t$ in probability as $n\to\infty$, by
\eqref{4.3}, {(A4)}, \eqref{eq:7}, and \eqref{eq:6a},
\begin{equation}
  \label{eq:4}
 \frac{D_i^n(t)}{n} \ra \mu_i(0)t
\,\mbox{ in probability as }\, n\ra\infty.
\end{equation}
The convergences in (A5), \eqref{eq:6}, (\ref{eq:6a}), and
\eqref{eq:4}
 imply if one
recalls the definitions in (\ref{eq:1a}), (\ref{eq:1b}),
and (\ref{eq:1c}) that the
$(M^{A,n}/\sqrt{n},M^{B,n}/\sqrt{n},M^{D,n}/\sqrt{n})$
converge in distribution to
$(M^A,M^B,M^D)$\,, where $M_i^A(t)=W^A_i(\lambda_i(0)t)$, $M_i^B(t)=
\sum_{j=1}^KW^\Phi_{ji}(\mu_j(0)t)$,
$M^D_i(t)=\sum_{j=1}^K(\delta_{ij}-p_{ji})W^D_j(\mu_j(0)t)$, so,
by (\ref{eq:1}) and \eqref{eq:3}, the $\overline{M}^{n}$
converge in distribution to $M$\,.
\end{proof}

\begin{proof}[\bf Proof of Theorem \ref{main}]
We note that by \eqref{2.7st},
\begin{equation}
  \label{eq:12}
  X^n(t)={\rm{\Gamma}}\left(X^n(0)+\inte{0}{\cdot}
\frac{1}{\sqrt{n}}\,\ab{a}{n}{}(\sqrt{n}\ab{X}{n}{}(s))ds
  +\overline{M}^n(\cdot)\right)(t), \quad \mbox{for } t\geq0.
\end{equation}
By the Lipschitz continuity of the {Skorohod} map,
(\ref{eq:2}), and (A2),
 for $T>0$ and suitable $H>0$,
\begin{multline*}
  \sup_{t\in[0,T]}|X^n(t)|\le |X^n(0)|+
L\inte{0}{t}
\frac{1}{\sqrt{n}}\,|\ab{a}{n}{}(\sqrt{n}\ab{X}{n}{}(s))|ds
  +\frac{1}{\sqrt{n}}\,\sup_{t\in[0,T]}|
 \ab{M}{n}{}(t)|\\
\le |X^n(0)|+
LH\inte{0}{t}(1+|X^n(s)|)\,ds
  +\frac{1}{\sqrt{n}}\,\sup_{t\in[0,T]}|
 \ab{M}{n}{}(t)|\,.
\end{multline*}
Gronwall's inequality, the convergence of the $X^n(0)$, and
Lemma \ref{lem2} yield
\begin{equation*}
  \lim_{r\to\infty}\limsup_{n\to\infty}
\mathbf{P}(\sup_{t\in[0,T]}|X^n(t)|>r)=0\,,
\end{equation*}
so, by (\ref{eq:2}) and (A4), the sequence of processes
$(\inte{0}{t}
\ab{a}{n}{}(\sqrt{n}\ab{X}{n}{}(s))/\sqrt{n}\,\,ds\,,t\ge0)$
is $\mathbb{C}$-tight.

By \eqref{eq:12}, the convergence of the $X^n(0)$, Lemma \ref{lem2}, (A4), Prohorov's theorem, and the continuity of the Skorohod map, the sequence of processes $(X^n(t),\,t\ge0)$ is $\mathbb{C}$-tight and every limit point $(\tilde{X}(t),\,t\ge0)$ for convergence in distribution satisfies the equation
\begin{equation*}
 \tilde X(t)={\rm{\Gamma}}\left(X(0)+\inte{0}{\cdot}
\,\ab{a}{}{}(\ab{\tilde X}{}{}(s))ds
  + \ab{M}{}{}(\cdot)\right)(t), \quad \mbox{for } t\geq0.
\end{equation*}
The uniqueness of a solution to the Skorohod problem implies that
 $\tilde{X}(t)=X(t)$\,.
\end{proof}
\section{Generalised Jackson networks with state-dependent rates}
\label{example}

In this section, we consider an application  to the setting of
 generalised Jackson networks in conventional scaling.
  Suppose as  given mutually independent
sequences of i.i.d. nonnegative random variables $\{u^i_j(n), i\geq1\}$,
$\{v^i_k(n), i\geq1\}$ for $j\in\mathcal J\subseteq \K$ and $k\in\K$.
For the $n$th network, the random variable $u^i_j(n)$ represents
 the $i$th exogenous interarrival time
 at station $j$,  while $v^i_k(n)$ is the $i$th service time at
 station $k$. The quantities $p_{ij}$ represent the
 probabilities of a job leaving station $i$  being routed directly
to station $j$,
 which are held constant. The routing decisions, interarrival and service
 times, and the initial queue length vector are mutually independent.

   We define \beqys && \mu_k^n = (\mathbf E[v^1_k(n)])^{-1}>0, \quad
   s_k^n=\mbox{\bf Var}(v^1_k(n))\geq0, \quad k\in\K, \quad \mbox{and} \\
   && \lambda_j^n = (\mathbf E[u^1_j(n)])^{-1}>0, \quad
   a_j^n=\mbox{\bf Var}(u^1_j(n))\geq0, \quad j\in\mathcal J, \eeqys with
   all of these terms assumed finite and the set $\mathcal J$
   nonempty. It is convenient to let $\lambda^n_j=1$ and
$a^n_j=0$ for $j\not\in \mathcal{J}$\,.

Let
$ \hat{N}^{{\hat A},n}_j(t)=\max\{i':\,\sum_{i=1}^{i'}u^i_j(n)\le t\}$ for
{$j\in\mathcal{J}$} and
$\hat{N}^{ {\hat {D}} ,n}_k(t)=\max\{i':\,\sum_{i=1}^{i'}v^i_k(n)\le t\}$ for $k\in\K$\,.
We may interpret the process $(\hat{N}^{{\hat A},n}_j(t),\,t\ge0)$ as a nominal
arrival process  and the random variables
$v^i_k(n)$ as the amounts of work needed to serve  the jobs.
Suppose that arrivals  are speeded up (or delayed) by a function
$\hat{\lambda}^n_i(x)$, where $i\in\mathcal{J}$,
and the service is performed at rate
$\hat{\mu}^n_k(x)$, where $k\in\K$,
 when the queue length vector is $x$.
As in Section \ref{mainresult}, we let
${\hat{N}^{\hat A,n}_i(t)=
\lfloor t\rfloor}$
and $\hat{\lambda}^n_i(x)=0$ for $i\not\in\mathcal{J}$.
In analogy
with (\ref{4})-\eqref{4.3} the queue lengths at the stations
at time $t$, which we
denote by $\hat{Q}^n_i(t)$, are assumed to satisfy the equations
  \begin{align*}
       \hat{Q}^n_i(t)=& \hat{Q}^n_i(0)
+ \hat{A}^n_i(t)+\hat{B}^n_i(t)-\hat{D}^n_i(t), \\
\hat{A}_i^n(t)=& \ab{\hat{N}}{\hat{A},n}{i}\left(\inte{0}{t}\ab{\hat{\lambda}}{n}{i}(\ab{\hat{Q}}{n}{}(s))ds\right), \\
\hat{B}_i^n(t)=& \sig{j=1}{K}\hat{\Phi}^n_{ji}\bigl(\hat{D}^n_j(t)\bigr),
\\ \hat{D}^n_i(t)=& \ab{\hat{N}}{\hat{D},n}{i}\left(\inte{0}{t}\ab{\hat\mu}{n}{i}(\ab{\hat{Q}}{n}{}(s))1_{\{\ab{\hat{Q}}{n}{i}(s)>0\}}ds
\right),
  \end{align*}
where
\begin{equation*}
  \hat{\Phi}^n_{ji}(m)=\sum_{l=1}^m \chi^n_{ji}(l),
\end{equation*}
with $\{(\chi^n_{ji}(l),\,i=1,2,\ldots,K),\,l=1,2,\ldots\}$ being
indicator random variables which are mutually independent for
different $j$ and $l$ and are such that
$\mathbf{P}(\chi^n_{ji}(l)=1)=p_{ji}$\,.

If we introduce the random variables $Q^n_i(t)=\hat{Q}^n_i(nt)$,
$A^n_i(t)=\hat{A}^n_i(nt)$, $B^n_i(t)=\hat{B}^n_i(nt)$,
$D^n_i(t)=\hat{D}^n_i(nt)$, $N^{A,n}_i(t)=\hat{N}^{\hat{A},n}_i(t/\lambda^n_i)$,
$N^{D,n}_i(t)=\hat{N}^{\hat{D},n}_i(t/\mu^n_i)$, and
$\Phi^n_{ji}(m)=\hat{\Phi}^n_{ji}(m)$, and functions
$\lambda^n_i(x)=n\lambda^n_i\hat{\lambda}^n_i(x)$ and
$\mu^n_i(x)=n\mu^n_i\hat{\mu}^n_i(x)$,
then we can see that they
satisfy equations \eqref{4}--\eqref{4.3}.
Condition (A0) holds as $N^{A,n}_i(t)/t\to 1$ and
$N^{D,n}_i(t)/t\to 1$ a.s. as $t\to\infty$\,.

If we also assume that $\hat{Q}^n(0)/\sqrt{n}\Ra X_0$, that, for
$\;k\in\K$
and $j\in\mathcal{J}$,
 \begin{align*} \mu_k^n\ra
\mu_k, \quad
s_k^n\ra s_k,\\
\lambda_j^n\ra\lambda_j,\quad a_j^n\ra a_j,
 \end{align*} as $n\ra\infty$\,,  and that
 \begin{equation*}\label{m2}\max_{k\in \K}\sup_{n\geq1}
\mathbf E(v^1_k(n))^{2+\eps}
+\max_{j\in \mathcal{J}}\sup_{n\geq1}
\mathbf E(u^1_j(n))^{2+\eps}  <
\infty \quad \mbox{for some }\,\, \eps>0,\end{equation*} then
condition (A5) holds with
$W^A_j=\sqrt{a_j}\lambda_j B^A_j$ for $j\in\mathcal{J}$,
$W^A_j(t)=0$ for $j\notin\mathcal{J}$, and
$W^D_k=\sqrt{s_k}\mu_kB^D_k$ for $k\in\K$\,, where
$B^A_j$ and $B^D_k$ are  independent standard Brownian motions,
with $\Phi_i$ being a $K$-dimensional Brownian motion with covariance matrix
$\mathbf{E}\Phi_{ik}(t)\Phi_{ij}(t)=(p_{ij}\delta_{jk}-p_{ij}p_{ik})t$,
 and with processes
$B^A_j$, $B^D_k$, and $\Phi_i$ being mutually independent.

Let us assume that the following versions of conditions (A2)--(A4) hold:
\begin{itemize}
 \item[$\widehat{(A2)}$] For each  $i\in\K$,
$$ \sup_{n\in\NN}\sup_{x\in
  S}\left(\frac{\hat{\lambda}^n_i(nx)}{1+|x|}
+ \frac{\hat{\mu}_i^n(nx)}{1+|x|} \right)<\infty,$$
\item[$\widehat{(A3)}$] There exist continuous
functions $\hat{\lambda}_i(x)$ and $ \hat{\mu}_i(x)$
such that $${\hat\lambda}^n_i(nx)\ra \hat{\lambda}_i(x),\quad
{\hat\mu}^n_i(nx)\ra \hat{\mu}_i(x)$$ uniformly on compact
subsets of $S$, as $n\ra\infty$. Furthermore, for $x \in S$,
$$\overline{\lambda}(x)-R\overline{\mu}(x)=0,$$
where $\overline{\lambda}_i(x)=\lambda_i\hat{\lambda}_i(x)$ and
$\overline{\mu}_i(x)=\mu_i\hat{\mu}_i(x)$\,,
\item[$\widehat{(A4)}$]
There exists a {Lipschitz}-continuous function $\hat{a}(x)$ such that
\begin{equation*}
\sqrt{n}(\overline{\lambda}^n(\sqrt{n}x)-
R\overline{\mu}^n(\sqrt{n}x))\to \hat{a}(x)
\end{equation*}
as $n\to\infty$ uniformly on compact subsets of $S$,
where $\overline{\lambda}^n_i(x)=\lambda^n_i\hat{\lambda}^n_i(x)$ and
$\overline{\mu}^n_i(x)=\mu^n_i\hat{\mu}^n_i(x)$\,{.}
\end{itemize}

Then the process $M$ in \eqref{sde} and \eqref{eq:3} is a $K$-dimensional
Brownian motion with
covariance matrix
 $\mathcal A$
which has  entries \begin{equation*}\label{a1} \mathcal A_{ii} =
\hat{\lambda}_i(0)\lambda_i^3
  a_i + \hat{\mu}_i(0)\mu_i^3 s_i(1-2p_{ii}) +
\sig{j=1}{K}\hat{\mu}_j(0)\mu_j p_{ji} (1-p_{ji} +
  p_{ji}\mu_j^2 s_j)\quad \mbox{for }\, i\in\K,  \end{equation*}
and \begin{equation*}\label{a2} \mathcal A_{ij}= -\left[
    \hat{\mu}_i(0)\mu_i^3
    s_ip_{ij} + \hat{\mu}_j(0)\mu_j^3 s_jp_{ji}
+\sig{k=1}{K} \hat{\mu}_k(0)\mu_kp_{ki} p_{kj}
    (1-\mu_k^2 s_k)\right] \quad \mbox{for }\, 1\leq i<j\leq
  K. \end{equation*}

An application of Theorem \ref{main} yields the following result.
\begin{cor}\label{renewal} If, in addition to the assumed hypotheses,
  condition  (A1) holds,
then the processes $(\hat{Q}^n(nt)/\sqrt{n},\,t\ge0)$ converge in
  distribution to the process $(X(t),\,t\ge0)$ with
  \begin{equation*}
    X(t)=\Gamma\Bigl(X_0 +   \inte{0}{\cdot}\ab{\hat a}{}{}(\ab{X}{}{}(s))ds +
 \mathcal{A}^{1/2}B(\cdot)\bigr)(t),
  \end{equation*}
where $B(\cdot)$ is a $K$-dimensional standard Brownian motion.
 \end{cor}
 \begin{remark}
   The conditions on the asymptotics of the arrival and service rates essentially boil
   down to the assumptions that
 the following expansions
   hold:
$\overline{\lambda}^n(x)=\overline{\lambda}_1(x/n)+
\overline{\lambda}_2(x/\sqrt n)/\sqrt{n}$ and
$\overline{\mu}^n(x)=\overline{\mu}_1(x/n)+
\overline{\mu}_2(x/\sqrt n)/\sqrt{n}$ with suitable functions
$\overline{\lambda}_1$, $\overline{\lambda}_2$,
$\overline{\mu}_1$, and $\overline{\mu}_2$.
 \end{remark}
 \begin{remark}
    If, in addition, the assumption of unit rates is made,
that is $\hat{\lambda}^n_j(x)=1$ for $j\in\mathcal{J}$
and $\hat{\mu}^n_k(x)=1$ for $k\in\K$,
then the limit process  is a $K$-dimensional reflected Brownian
motion on the positive orthant with infinitesimal drift $\hat{a}(0)$ and
covariance matrix $\mathcal A$, and the  reflection matrix $R=I-P^T$,
as in Theorem 1 of Reiman \cite{Reiman:1984}.
 \end{remark}
\begin{remark}
In order to extend  applicability,
one may consider independent sequences of weakly dependent random
variables $\{u^i_j(n), i\geq1\}$, $\{v^i_k(n), i\geq1\}$ for
$j\in\mathcal J\subseteq \K$ and $k\in\K$.  Under suitable moment
and mixing conditions which imply  the invariance principle, cf.,
e.g.,
Herrndorf \cite{Herrndorf1984}, Peligrad \cite{Peligrad1986}, Jacod
and Shiryaev \cite{MR2003j:60001},  Corollary \ref{renewal} continues to hold.
\end{remark}

\renewcommand{\theequation}{{$\mathbb{A}$\arabic{equation}}}
\setcounter{equation}{0}  
\section*{Appendix}
\begin{proof}[\bf Proof of Lemma \ref{exist}]
The proof is an adaptation of the one in Puhalskii and Simon
\cite[Lemma 2.1]{PuhSim12} and employs the approach of
Ethier and Kurtz \cite[Theorem 4.1,
p.327]{EthKur86}.
Let
\begin{align*}
  \theta^n(x)&=1+\sum_{i=1}^K(\mu_i^n(x)+\lambda_i^n(x))\,,\\
\hat{\mu}_i^n(x)&=\frac{\mu^n_i(x)}{\theta^n(x)},\\
\hat{\lambda}_i^n(x)&=\frac{\lambda^n_i(x)}{\theta^n(x)}\,{,}
\end{align*}
and
\begin{equation*}
  \tau^n(t)=\inf\{s:\,   \int_0^s\theta^n(Q^n(u))\,du>t\}\,.
\end{equation*}
We note that $\tau^n(t)$ is finite-valued, differentiable,
$d\tau^n(t)/dt=1/\theta^n(Q^n(\tau^n(t))$\, and $\tau^n(t)\to\infty$
as $t\to\infty$\,.
One can see that if the process $Q^n$ satisfies a.s. the equations
\begin{align}
  \label{eq:13}
\notag          Q^{n}_{i}(t)&= Q^n_{i}(0)+
\ab{N}{A,n}{i}\Bigl(\inte{0}{t}\ab{\lambda}{n}{i}(\ab{Q}{n}{}(s))ds\Bigr)\\&\notag
+\sig{j=1}{K}\Phi^n_{ji}\Bigl(\ab{N}{D,n}{j}\Bigl(\inte{0}{t}\ab{\mu}{n}{j}(\ab{Q}{n}{}(s))1_{\{\ab{Q}{n}{j}(s)>0\}}ds
\Bigr)\Bigr)
\\&-\ab{N}{D,n}{i}\Bigl(\inte{0}{t}\ab{\mu}{n}{i}(\ab{Q}{n}{}(s))1_{\{\ab{Q}{n}{i}(s)>0\}}ds
\Bigr)\,,t\ge0,
\end{align}
then the process $\hat{Q}^n=(\hat{Q}^n(t)\,,t\ge0)$
defined by $\hat{Q}^n(t)=Q^n(\tau^n(t))$ satisfies a.s. the equations
\begin{align}
  \label{eq:16}
\notag          \hat{Q}^{n}_{i}(t)&= \hat{Q}^n_{i}(0)+
\ab{N}{A,n}{i}\Bigl(\inte{0}{t}
\ab{\hat\lambda}{n}{i}(\ab{\hat{Q}}{n}{}(s))
ds\Bigr)\\&\notag
+\sig{j=1}{K}\Phi^n_{ji}\Bigl(\ab{N}{D,n}{j}
\Bigl(\inte{0}{t}\ab{\hat\mu}{n}{j}(\ab{\hat{Q}}{n}{}(s))1_{\{\ab{\hat{Q}}{n}{j}(s)>0\}}ds
\Bigr)\Bigr)
\\&-\ab{N}{D,n}{i}\Bigl(\inte{0}{t}\ab{\hat\mu}{n}{i}(\ab{\hat{Q}}{n}{}(s))1_{\{\ab{\hat{Q}}{n}{i}(s)>0\}}ds
\Bigr)\,,t\ge0.
\end{align}
On the other hand, suppose a $\Z_+^K$-valued process
$\hat{Q}^n$ satisfies a.s. \eqref{eq:16} and let
\begin{equation*}
\hat{\tau}^n(t)=\inf\{s:\,
\int_0^s\frac{1}{\theta^n(\hat{Q}^n(u))}\,du>t\}\,.
\end{equation*}
We show that $\hat{\tau}^n(t)$ is well defined for all $t$ a.s.
Since by condition (A2), for a suitable constant $L^n$, $\theta^n(x)\le L^n(1+x)$, we have
that
  \begin{equation}
    \label{eq:14}
  \int_0^s\frac{1}{\theta^n(\hat{Q}^n(u))}\,du\ge
\frac{1}{L^n}\,\int_0^s\frac{1}{1+\sum_{i=1}^K\hat{Q}_i^n(u)}\,du
\ge \frac{1}{L^n}\,\int_0^s\frac{1}{1+\sum_{i=1}^K\hat{Q}_i^n(0)+\sum_{i=1}^KN^{A,n}_i(u)}\,du,
\end{equation}
where the latter inequality uses  the fact that by \eqref{eq:16}
\begin{equation*}
\sum_{i=1}^K\hat{Q}_i^n(t)\le \sum_{i=1}^K\hat{Q}_i^n(0)+
\sum_{i=1}^KN^{A,n}_i\bigl(\int_0^s\hat{\lambda}^n_i(\hat{Q}^n_i(u))\,du\bigr)
\end{equation*}
and that $\hat{\lambda}^n_i(x)\le 1$\,. Since
$\limsup_{t\to\infty}N^{A,n}_i(t)/t<\infty $ a.s., the rightmost integral in \eqref{eq:14} tends to
infinity as $t\to\infty$ a.s., so does the leftmost integral, which proves
the claim.
In addition, $\hat{\tau}^n(t)$ is  differentiable,
$d\hat{\tau}^n(t)/dt=\theta^n(\hat{Q}^n(\hat\tau^n(t))$\,
and $\hat\tau^n(t)\to\infty$
as $t\to\infty$ a.s. It follows that $Q^n(t)=\hat{Q}^n(\hat{\tau}^n(t))$
satisfies \eqref{eq:13} a.s.

Thus, existence and uniqueness for (\ref{eq:13}) holds if and only if
existence and uniqueness holds for (\ref{eq:16}).
The existence and uniqueness for \eqref{eq:16} follows by recursion on
the jump times of $\hat{Q}^n$.
In some more detail, we define  the processes $\hat{Q}^{n,\ell}=(\hat{Q}^{n,\ell}(t),\,t\ge0)$ with
$\hat{Q}^{n,\ell}(t)=(\hat{Q}^{n,\ell}_i(t),\,i=1,2,\ldots,K)$ by
 $\hat{Q}^{n,0}_{i}(t)=\hat{Q}^n_{i}(0)$ and, for $\ell=1,2,\ldots$,  by
 \begin{align*}
\notag          \hat{Q}^{n,\ell}_{i}(t)&= \hat{Q}^n_{i}(0)+
\ab{N}{A,n}{i}\Bigl(\inte{0}{t}\ab{\hat{\lambda}}{n}{i}(\ab{\hat{Q}}{n,\ell-1}{}(s))ds\Bigr)\\&\notag
+\sig{j=1}{K}\Phi^n_{ji}\Bigl(\ab{N}{D,n}{j}\Bigl(\inte{0}{t}\ab{\hat{\mu}}{n}{j}(\ab{\hat{Q}}{n,\ell-1}{}(s))1_{\{\ab{\hat{Q}}{n,\ell-1}{j}(s)>0\}}ds
\Bigr)\Bigr)
\\&-\ab{N}{D,n}{i}\Bigl(\inte{0}{t}\ab{\hat{\mu}}{n}{i}(\ab{\hat{Q}}{n,\ell-1}{}(s))1_{\{\ab{\hat{Q}}{n,\ell-1}{i}(s)>0\}}ds
\Bigr)\,.
\end{align*}
Let $\tau^{n,\ell}$ represent the time epoch of the $\ell$th jump of
$\hat{Q}^{n,\ell}$ with $\tau^{n,0}=0$\,.
One can see that $\hat{Q}^{n,1}(t)=\hat{Q}^{n,0}(0)$ if
$t<\tau^{n,1}$\,. It follows that
 $(\hat{Q}^{n,1}(t),\,t\ge0)$ and $(\hat{Q}^{n,2}(t),\,t\ge0)$
experience the first jump at the same time epoch and the jump size is
the same for both processes, so $\tau^{n,1}<\tau^{n,2}$ and
$\hat{Q}^{n,1}(t\wedge
\tau^{n,1})=\hat{Q}^{n,2}(t)$ for $t<\tau^{n,2}$\,.
We define $\hat{Q}^n(t)=\hat{Q}^n(0)$ for $t<\tau^{n,1}$ and
$\hat{Q}^n(t)=\hat{Q}^{n,1}(t)$ for $\tau^{n,1}\le t<\tau^{n,2}$\,.
Similarly, for an arbitrary $\ell\in\NN$,  we obtain
that $\tau^{n,\ell}<\tau^{n,\ell+1}$ and
$\hat{Q}^{n,\ell}(t\wedge
\tau^{n,\ell})=\hat{Q}^{n,\ell+1}(t)$ for $t<\tau^{n,\ell+1}$\,.
We let
$\hat{Q}^n(t)=\hat{Q}^{n,\ell}(t)$ for $\tau^{n,\ell}\le
t<\tau^{n,\ell+1}$\,. The process $\hat{Q}^n$ is defined consistently
for $t\in \cup_{\ell=1}^\infty[\tau^{n,\ell-1},\tau^{n,\ell})$\,. If
$\tau^{n,\ell+1}=\infty$ for some $\ell$, then we let
$\hat{Q}^n(t)=\hat{Q}^{n,\ell}(t)$ for all $t\ge\tau^{n,\ell}$\,.

Suppose that $\tau^{n,\ell}<\infty$ for all $\ell$\,. Then
 $\hat{Q}^n(t)$  has been defined
for all
$t<\tau^{n,\infty}=\lim_{\ell\to\infty}\tau^{n,\ell}$ and satisfies
\eqref{eq:16}
for these values of $t$\,.
We now show that
$\tau^{n,\infty}=\infty$\,. The set of the time epochs of the jumps of
$\hat{Q}^n$ is a subset of the set of the time epochs of the jumps of
the process $\tilde{Q}^n=(\tilde{Q}^n(t),\,t\ge0)$, where
\begin{multline*}
  \tilde{Q}^n(t)=
\sum_{i=1}^K\biggl(\ab{N}{A,n}{i}\Bigl(\inte{0}{t}\ab{\hat{\lambda}}{n}{i}(\ab{\hat{Q}}{n,\ell-1}{}(s))ds\Bigr)
+\sig{j=1}{K}\Phi^n_{ji}
\Bigl(\ab{N}{D,n}{j}\Bigl(\inte{0}{t}\ab{\hat{\mu}}{n}{j}(\ab{\hat{Q}}{n,\ell-1}{}(s))1_{\{\ab{\hat{Q}}{n,\ell-1}{j}(s)>0\}}ds
\Bigr)\Bigr)
\\+\ab{N}{D,n}{i}\Bigl(\inte{0}{t}\ab{\hat{\mu}}{n}{i}(\ab{\hat{Q}}{n,\ell-1}{}(s))1_{\{\ab{\hat{Q}}{n,\ell-1}{i}(s)>0\}}ds
\Bigr)\biggr)\,.
\end{multline*}
Since the process $\hat{Q}^n$ has infinitely many jumps, so does the
process
$\tilde{Q}^n$\,.
Since $\hat{\lambda}^n_i(x)\le1$, $\mu^n_i(x)\le 1$
and $\Phi^n_{ji}(m_1)-\Phi^n_{ji}(m_2)\le m_1-m_2$
for $m_1\ge m_2$, the lengths of time between the jumps of
$\tilde{Q}^n$ are not less than the lengths of time between the
corresponding jumps of the process
$\bigl(\sum_{i=1}^K\ab{N}{A,n}{i}(t)
+\sig{i=1}{K}\ab{N}{D,n}{i}(t),\,t\ge0\bigr)\,.$ The process
$\bigl(\sum_{i=1}^K\ab{N}{A,n}{i}(t)
+\sig{i=1}{K}\ab{N}{D,n}{i}(t),\,t\ge0\bigr)$ having {infinitely} many
jumps, the time
epochs of the jumps of $\bigl(\sum_{i=1}^K\ab{N}{A,n}{i}(t)
+\sig{i=1}{K}\ab{N}{D,n}{i}(t),\,t\ge0\bigr)$ tend to infinity as the
jump numbers tend to infinity.  Thus, $\tau^{n,\infty}=\infty$ a.s.

The provided construction shows that $Q^n$ is a suitably
measurable function of
$N^{A,n}$, $N^{D,n}$, and $\Phi^n$, so it
is a strong solution.
We have  proved the existence of a strong solution to \eqref{eq:16}\,. A
similar argument establishes uniqueness.
\end{proof}

\def\cprime{$'$} \def\cprime{$'$} \def\cprime{$'$} \def\cprime{$'$}
  \def\polhk#1{\setbox0=\hbox{#1}{\ooalign{\hidewidth
  \lower1.5ex\hbox{`}\hidewidth\crcr\unhbox0}}} \def\cprime{$'$}
  \def\cprime{$'$}

\end{document}